\documentclass[preprint,12pt]{elsarticle}
\usepackage{latexsym}
\usepackage{amssymb}
\usepackage{amsmath}
\usepackage{amsthm}
\usepackage{verbatim}

\numberwithin{equation}{section}

\newtheorem{theorem}{Theorem}[section]
\newtheorem{lemma}{Lemma}[section]
\newtheorem{proposition}{Proposition}[section]
\newtheorem{remark}{Remark}[section]
\newtheorem{definition}{Definition}[section]

\journal{Journal}

\begin{document}

\begin{frontmatter}



\title{Periodic solutions of a semilinear Euler-Bernoulli beam equation with variable coefficients
}


\author[ad1]{Hui Wei}
\ead{weihui01@163.com}
\author[ad2,ad3]{Shuguan Ji\corref{cor}}
\ead{jishuguan@hotmail.com}
\address[ad1]{Department of Mathematics, Luoyang Normal University, Luoyang 471934, P.R. China}
\address[ad2]{School of Mathematics and Statistics and Center for Mathematics and Interdisciplinary Sciences, Northeast Normal University, Changchun 130024, P.R. China}
\address[ad3]{School of Mathematics, Jilin University, Changchun 130012, P.R. China}
\cortext[cor]{Corresponding author.}

\begin{abstract}
This paper is devoted to the study of periodic solutions for a semilinear Euler-Bernoulli beam equation with variable coefficients. Such mathematical model may be described the infinitesimal, free, undamped in-plane bending vibrations of a thin straight elastic beam. When the frequency $\omega =\frac{2\pi}{T}$ is rational, some properties of the beam operator with variable coefficients are investigated. We obtain the existence of periodic solutions when the nonlinear term is monotone and bounded.
\end{abstract}


\end{frontmatter}


\section{Introduction}

In this paper, we are concerned with the existence of periodic solutions to the following Euler-Bernoulli beam equation
\begin{equation}\label{eq1-1}
\rho(x)u_{tt} + (\eta(x)u_{xx})_{xx}+\rho(x)g(u)=f(t,x), \,  t\in \mathbb{R},\, 0<x<\pi,
\end{equation}
with the boundary conditions
\begin{equation}\label{eq1-2}
u(t,x)= 2(\alpha(x)+\beta(x))u_x(t,x)+ u_{xx}(t,x)=0, \ {\rm at} \ x=0 \ {\rm and} \ x=\pi,
\end{equation}
and with the periodic conditions
\begin{equation}\label{eq1-3}
u(t+T, x) = u(t, x),
\end{equation}
where $f(t,x)$ is a $T$-periodic function in time $t$ and the period $T$ satisfies
\begin{equation}\label{eq1-4}
T = 2\pi \frac{p}{q},    \ \ {\rm for} \ \ p, q \in \mathbb{N}^+ \ \ {\rm and} \ \ {\rm GCD} \ (p, q) =1.
\end{equation}

The equation \eqref{eq1-1} originated from the Euler-Bernoulli beam equation
\begin{equation*}
\rho(x)u_{tt} + (E\eta(x)u_{xx})_{xx}=0,
\end{equation*}
which can be used to account for the infinitesimal, free, undamped in-plane bending vibrations of a thin straight elastic beam. Here, $\rho$ is the linear mass density, $E$ is Young's modulus of the material, $\eta$ is the second moment of
area of the beam's cross-section and the product $E\eta$ is the bending stiffness. In this paper, $E$ is assumed to be a constant.

 The boundary conditions \eqref{eq1-2} mean that the two ends of a beam are constrained by some special rotational spring devices. In particular, when $\alpha(0)+\beta(0)=\alpha(\pi)+\beta(\pi)=0$, this boundary conditions correspond to the pinned-pinned one
\begin{equation}\label{eq1-a}
u(t,0)=u_{xx}(t,0)=u(t,\pi)=u_{xx}(t,\pi)=0.
\end{equation}


For the case  $\rho(x)=\eta(x)=C$ (a non-zero constant),  the equation \eqref{eq1-1} corresponds to the classical beam equation which is called the constant coefficient beam equation here. For the space dimension $d=1$ and the frequency $\omega =\frac{2\pi}{T}$ is rational, the existence of periodic solutions to the nonlinear beam equations with constant coefficient has been studied by many authors (see \cite{CS82,F88, L95,Ru12, Ru15}). It is well known that the variational method pioneered by Rabinowitz \cite{Ra67,Ra78} is a powerful tool for dealing with nonlinear problems, and it is closely related to compactness.
Recently, by variational method, Rudakov \cite{Ru18} proved that, under the boundary conditions \eqref{eq1-a}, there is a sequence of periodic solutions to the nonlinear beam equation
\begin{equation*}
u_{tt}+u_{xxxx}-cu_{xx}+g(t,x,u)=f(t,x), \ \ 0<x<\pi,
\end{equation*}
where $c$ satisfies the condition (A) in  \cite{Ru18} and the nonlinear term $g$ has a polynomial growth in $u$.
 The condition (A) can make sure that the inverse of the linearized operator of this problem is compact on its range.
In addition, the above works \cite{CS82,F88, L95,Ru12, Ru15} also adopted  variational method  because of the properties of the operator
\begin{equation*}
\mathcal{L} :=\frac{\partial^2 }{\partial t^2} +\frac{\partial^4 }{\partial x^4}.
\end{equation*}
More precisely, since the eigenvalues of $\mathcal{L}$ possess the explicit expression
\begin{equation}\label{eq1-5}
n^4 -(\omega m)^2, \ \  n\in \mathbb{N}^+, \  m\in \mathbb{Z},
\end{equation}
then  $0$ is the only eigenvalues with infinite multiplicity and the other eigenvalues have finite multiplicity and accumulate to infinity,  which implies that $\mathcal{L}^{-1}$ (the inverse of $\mathcal{L}$) is compact on its range.
For the case that the frequency $\omega$ is irrational, from \eqref{eq1-5}, it follows that $0$ is not the eigenvalues, but it is an accumulation point of the spectrum, which caused the ``small divisor" problem. A further investigation shows $\mathcal{L}^{-1}$ is not compact, thus variational method is invalid to deal with the problem.
KAM theory \cite{K87,Wa90} and Nash-Moser iteration technique \cite{B95,C93} are derived to solve such problem. Since  KAM theory is more efficient in finding quasi-periodic solutions, it was widely applied to the nonlinear beam equations with constant coefficient (see \cite{CG18,Sh16,S19,W12,W19}).

For the space dimension $d>1$, recently, Wang and Li \cite{W18} studied the beam equation with weak damping
\begin{equation*}
u_{tt}+\Delta^2u -\Delta u + u_t  =\Delta \Phi(u)+ \Delta\Psi(t,x), \ \ x\in \mathbb{R}^d,
\end{equation*}
where the nonlinear term $\Phi$ is a
smooth function  satisfying $\Phi(u)=O (u^2)$ as $u\rightarrow 0$. They obtained the existence and uniqueness of  periodic solution via  contraction mapping theorem. In \cite{CL18}, by  a para-differential method, Chen et al. got  a family of periodic solutions to the beam equation
\begin{equation*}
u_{tt} + \Delta^2u +\kappa u =\varepsilon \frac{\partial F}{\partial u}(\omega t, x,u,\varepsilon) + \varepsilon f(\omega t, x), \ \ x\in \Bbb T^d:= (\mathbb{R}/2\pi \mathbb{Z})^d,
\end{equation*}
where  the nonlinear term
$F$ is a $2\pi$-periodic smooth function  and satisfies $\partial^k_\xi F(t, x,\xi,\varepsilon)\mid_{\xi=0}\equiv 0$ for $k \leq 2$.

The nonlinear  PDEs with variable coefficients (i.e., $\rho(x)$ and $\eta(x)$ are two functions) has received great attention due to its wide application. For example,  the forced vibrations of a bounded nonhomogeneous string and the propagation of seismic waves in non-isotropic media are governed by the variable coefficients wave equation (see \cite{Bar96,Bar97a,Bar97b})
\begin{equation}\label{eq1-b}
\rho(x)u_{tt} - (\eta(x)u_{x})_{x}=0.
\end{equation}
Here, $\rho(x)$ and $\eta(x)$ respectively denote the rock density and the elasticity coefficient. Under some homogeneous boundary conditions, the existence of the periodic solutions to the equation \eqref{eq1-b} equipped with various nonlinearity has been extensively studied  in recent decades (see \cite{BB06,Bar97a, Chen15,Chen19, Ji08,Ji18,Ji11,Ma18,Ma19a,Ru17,W.19,WJ.19}).

In comparison with the variable coefficients wave equation, the periodic solutions of the nonlinear beam equation with variable coefficients were rarely studied. As far as we know,  the only work \cite{ChenXiv} is Chen et al.'s research on the beam equation
\begin{equation*}
\rho(x)u_{tt} + (\eta(x)u_{xx})_{xx}=f(\omega t,x,u), \, \ t\in \mathbb{R}, \  0<x<\pi, \\
\end{equation*}
together with the boundary conditions \eqref{eq1-a}.
They got the periodic solution of this problem  via Nash-Moser iteration technique. 
The purpose of this  paper is  to investigate the existence of the periodic solution to the problem \eqref{eq1-1}--\eqref{eq1-3} by variational method when $\omega$ is  rational. Compared with \cite{ChenXiv}, although we do not have to deal with the ``mall divisor'' problem anymore, several new challenges will arise: (i) The properties of the operator $L$ defined by
\begin{equation*}
L= \frac{\partial^2 }{\partial t^2} +  \rho^{-1}\frac{\partial^2 }{\partial x^2}\Big{(}\eta \frac{\partial^2}{\partial x^2}\Big{)}
\end{equation*}
have never been studied. Particularly, when we use Lemma \ref{lem3-1} (see Section \ref{sec:3}) to find periodic solutions, it requires that $L^{-1}$ is compact on its range, but it is't always the case from Remark \ref{re2-1} (see Section \ref{sec:2}). Thus a difficulty encountered is what conditions can guarantee that $L^{-1}$ is compact;
(ii) The eigenvalues of $L$ can only be given the asymptotical expression rather than the explicit expression like $\mathcal{L}$. As a result, it makes the study of the properties of $L$ more difficult.

The present paper aims to prove the existence of the periodic solutions to the problem \eqref{eq1-1}--\eqref{eq1-3} by Lemma \ref{lem3-1} when the frequency $\omega$ is rational. We investigate some properties of $L$, especially we prove $L^{-1}$ is compact on its range under some suitable conditions. Thereafter, when the nonlinear term $g$ is monotone and bounded, we obtain the existence of the periodic solution.

The rest of this paper is organized as follows. In Sect. \ref{sec:2}, we give
the definition of the weak solution of the problem \eqref{eq1-1}--\eqref{eq1-3} and study  some properties of the beam operator with variable coefficients. In  Sect. \ref{sec:3}, we introduce the Lemma \ref{lem3-1} owed to Br\'{e}zis and Nirenberg \cite{BN78} and use it to prove the main result Theorem \ref{th3-1}.

\section{Beam operator with variable coefficients}

\setcounter{equation}{0}

\label{sec:2}

Let $\Omega = (0, T) \times (0, \pi)$ and
\begin{eqnarray*}
\Psi = \{\psi \in C^\infty(\Omega) : \psi \ {\rm satisfies \ the \ boundary \ conditions} \ \eqref{eq1-2}, \\
\psi(0,x) = \psi(T,x) \ {\rm and} \ \psi_t(0,x) = \psi_t(T,x).\}.
\end{eqnarray*}
Set
\begin{equation*}
L^r(\Omega) = \Big\{ u: \|u\|^r_{L^r(\Omega)} = \int_\Omega |u(t,x)|^r \rho(x) \textrm{d}t \textrm{d}x <\infty\Big\}, \, \, r\geq 1.
\end{equation*}
Obviously, $\Psi$ is dense in $L^r(\Omega)$ for any $r\geq 1$,  $L^2(\Omega)$ is a Hilbert space with the inner product
\begin{equation*}
\langle u, v \rangle = \int_\Omega u(t,x)  \overline{v(t,x)} \rho(x)\textrm{d}t \textrm{d}x, \ \forall u, v \in L^2(\Omega),
\end{equation*}
and its norm is denoted by $\|\cdot\|$ for simplicity. In particular, the norm of $L^\infty(\Omega)$ is given by
\begin{equation*}
\|u\|_{L^\infty(\Omega)} = {\rm ess} \, \sup \{|u(t,x)|: (t,x)\in \Omega\}, \ \forall u \in L^\infty(\Omega).
\end{equation*}

Since $f(t,x)$ is a $T$-periodic function, it's sufficient to consider the  problem \eqref{eq1-1}--\eqref{eq1-3} in $\Omega$, and we rewrite it as follows
\begin{equation}\label{eq2-1}
\rho(x)u_{tt} + (\eta(x)u_{xx})_{xx}+\rho(x)g(u)=f(t,x), \,  \ (t, x) \in \Omega,
\end{equation}
with the boundary conditions
\begin{equation}\label{eq2-2}
u(t,x)= 2(\alpha(x)+\beta(x))u_x(t,x)+u_{xx}(t,x)=0, \ {\rm at} \ x=0 \ {\rm and} \ x=\pi,
\end{equation}
and with the periodic conditions
\begin{equation}\label{eq2-3}
u(0, x) = u(T, x), \ \ u_t(0, x) = u_t(T, x).
\end{equation}

\begin{definition}
Let $f\in L^2(\Omega)$, then a function $u \in L^2(\Omega)$ is said to be a weak solution to the  problem \eqref{eq2-1}--\eqref{eq2-3} if it satisfies
\begin{equation*}
\int_\Omega u(\rho\psi_{tt} + (\eta\psi_{xx})_{xx}){\rm d}t {\rm d}x + \int_\Omega g(u)\psi\rho{\rm d}t {\rm d}x= \int_\Omega f\psi {\rm d}t {\rm d}x, \ \  \forall \psi \in \Psi.
\end{equation*}
\end{definition}

The beam operator with variable coefficients is defined by
$$L \psi = \rho^{-1}\left(\rho \psi_{tt} + (\eta \psi_{xx})_{xx}\right), \ \forall \psi \in \Psi.$$
Clearly, $L$ is a linear operator.
Because we are interested in the weak solutions, the domain of $L$ is defined by
\begin{eqnarray*}
D(L)=\Big{\{}u\in L^2(\Omega): {\rm there \ exists} \ h \in L^2(\Omega) \ {\rm such \ that}  \quad \quad \quad\\
\int_\Omega u(\rho\psi_{tt} + (\eta\psi_{xx})_{xx}){\rm d}t {\rm d}x = \int_\Omega h\psi {\rm d}t {\rm d}x, \ \  \forall \psi \in \Psi.\Big{\}}.
\end{eqnarray*}
The operator $L$ maps $L^2(\Omega)$ into itself and is symmetric under the boundary conditions \eqref{eq2-2}.
Furthermore, noting that $\Psi \subset D(L)$ and is dense in $L^2(\Omega)$, one can verify that $L$ is a closed, dense, self-adjoint operator.

To  study the  periodic solutions of the problem \eqref{eq2-1}--\eqref{eq2-3}, we first pay an investigation to the eigenvalues and eigenfunctions of $L$.

It is well known that  the complete orthonormal system
\begin{equation*}
\phi_m(t) = T^{-\frac{1}{2}} e^{i\vartheta_m t}, \ \vartheta_m = 2\pi mT^{-1}, \ m \in \mathbb{Z},
\end{equation*}
forms a basis of $L^2(0,T)$.

Now, we consider the eigenvalue problems associated with the Euler-Bernoulli operator
\begin{equation*}
\mathcal{E}\varphi_n(x):=\frac{1}{\rho(x)}(\eta(x)\varphi''_n(x))''=\mu_n\varphi_n(x), \ \ n\in \mathbb{N}^+,
\end{equation*}
with the boundary conditions
\begin{equation*}
\varphi_n(x)=2(\alpha(x)+\beta(x))\varphi'_n(x)+\varphi''_n(x)=0, \ {\rm at} \ x=0 \ {\rm and} \ x=\pi,
\end{equation*}
where $\varphi'_n(x)= \frac{{\rm d}}{{\rm d}x}\varphi_n(x)$.

In order to characterize the eigenvalues of $\mathcal{E}$, it needs  the following  assumptions.

(A1) Let $\alpha(x)$, $\beta(x)$ have real value and belong to $W^3(0, \pi)$ which denotes the usual Sobolev space, i.e.,
\begin{equation*}
W^3(0, \pi)=\Big{\{} y\in L^1(0,\pi): y^{(3)} \ {\rm exists \ in \ the \ weak \ sense \ and \ belonges  \ to} \ L^1(0,\pi). \Big{\}}.
\end{equation*}

(A2) Let $\eta(0) =1$, $\rho(x)$ and $\eta(x)$ satisfy
\begin{equation*}
\rho(x) = \rho(0)e^{4\int_0^x \beta(s){\rm d}s}>1, \ \ \ \ \eta(x) = e^{4\int_0^x \alpha(s){\rm d}s},
\end{equation*}
and  let
\begin{equation*}
\int_0^\pi \Big( \frac{\rho}{\eta}\Big)^{\frac{1}{4}}{\rm d}x=\pi.
\end{equation*}

Based on the assumptions (A1) and (A2), by  the unitary Barcilon-Gottlieb transformation \cite{B87,G87} (namely, the Liouville-type transformation for a fourth-order operator), Badanin and Korotyaev \cite{BK15} obtained
\begin{equation}\label{eq2-4}
\mu_n = n^4 +2n^2a + b_n +o\big(\frac{1}{n}\big),
\end{equation}
where $a$ is a constant and $b_n$ is bounded independent of $n$. In fact, Badanin and Korotyaev considered the problem on the interval $[0,1]$, by the scale transformation $x\mapsto  \pi x$, we got the asymptotical formula \eqref{eq2-4}.
In addition, due to the appearance of $\exp(\int_0^\pi \beta(s)-\alpha(s){\rm d}s)$ in the expression $a$ (see \cite{BK15}), it is reasonable to assume that $a$ is irrational in the Proposition \ref{pro2-1}.


From \eqref{eq2-4}, $\mu_n$ has finite multiplicity  and accumulates to infinity, then  we can find a subsequence $\{\varphi_{n_k}(x)\}$ which constitutes a complete orthonormal basis of $L^2(0,\pi)$. In fact, on the one hand the eigenfunctions  corresponding to the different eigenvalues are always orthogonal. On the other hand, for the case that the eigenfunctions  correspond to the same eigenvalues, if necessary, we carry out Gram-Schmidt orthogonalization procedure to get the orthogonal eigenfunctions.
For the sake of convenience, we still use $\{\varphi_{n}(x)\}$ to denote $\{\varphi_{n_k}(x)\}$.

Moreover, we add the normalization
\begin{equation*}
\|\varphi_n\|^2_{L^2(0,\pi)}=\int_0^\pi \varphi^2_n(x) \rho(x) {\rm d}x =1.
\end{equation*}
In virtue of $\rho(x)>1$, we have $|\varphi_n(x)| <1, \ \forall x\in(0,\pi)$ and $n\in \mathbb{N}^+$.

Summarizing the above discussion, the eigenfunction system $\{\phi_m(t)\varphi_n(x)\}$ forms a complete orthonormal basis of $L^2(\Omega)$. In addition, it is easy to see the eigenvalues of $L$ have the form
\begin{equation*}
\lambda_{mn}=\mu_n-\vartheta^2_m.
\end{equation*}
Denote the spectrum of $L$ by
\begin{equation*}
\Lambda(L)=\{\lambda_{mn}:\lambda_{mn}=\mu_n-\vartheta^2_m\}.
\end{equation*}

Thus, the null space $N(L)$ of $L$ is given by
\begin{equation*}
N(L)= {\rm span} \{\phi_m(t)\varphi_n(x): \mu_n=\vartheta^2_m, \ m \in \mathbb{Z}, n\in \mathbb{N}^+\},
\end{equation*}
and its orthogonal complement is
\begin{equation*}
N(L)^\perp= {\rm span} \{\phi_m(t)\varphi_n(x): \mu_n\neq\vartheta^2_m, \ m \in \mathbb{Z}, n\in \mathbb{N}^+\}.
\end{equation*}

\begin{proposition}\label{pro2-1}
Let $T$ satisfy \eqref{eq1-4}, and let {\rm (A1)} and {\rm (A2)} hold. 
If $a$ is an irrational number, then $\dim N(L)< \infty$ and $\Lambda(L)$ is an unbounded discrete set.
\end{proposition}
\begin{proof}



In virtue of
\begin{equation*}
\lambda_{mn}=\mu_n-\vartheta^2_m,
\end{equation*}
we have
\begin{equation}\label{eq2-e}
\lambda_{mn}=\frac{1}{p^2}(pn^2+qm+pa)(pn^2-qm+pa)+b_n-a^2+o\big(\frac{1}{n}\big),
\end{equation}
 Letting $\lambda_{mn} =0$, a direct calculation shows
\begin{equation}\label{eq2-5}
(pn^2+qm+pa)(pn^2-qm+pa)=p^2(a^2-b_n)+o\big(\frac{1}{n}\big).
\end{equation}

Since $a$ is irrational, the left side of \eqref{eq2-5} accumulate to infinity, i.e.,
\begin{equation*}
pn^2+q|m|+pa \rightarrow \infty, \ \ |m|, \, n \rightarrow \infty.
\end{equation*}
Moreover, since $b_n$ is bounded independent of $n$, then the right side of \eqref{eq2-5} is bounded independent of $n$. Therefore, at most a finite number of pairs $(m,n)$ satisfy \eqref{eq2-5}, which implies $\dim N(L)< \infty$.

Using again that $a$ is irrational, it follows, from \eqref{eq2-e}
\begin{equation}\label{eq2-a}
\lambda_{mn} \rightarrow \infty, \ \ {\rm as} \ |m|, n \rightarrow \infty,
\end{equation}
 which implies that $\Lambda(L)$ is an unbounded discrete set. The proof is completed.
\end{proof}

\begin{remark}\label{re2-1}
In this paper, we study the existence of the periodic solutions  to the problem \eqref{eq2-1}--\eqref{eq2-3} by Lemma \ref{lem3-1} which requires the compactness of $L^{-1}$. In this remark and the Proposition \ref{pro2-2}, we will show that the compactness of $L^{-1}$ may be related to the arithmetic properties of $a$.

Recalling that $\{b_n\}$ is a bounded sequence, we can extract a subsequence  $\{b_{n_k}\}$ such that $b_{n_k}\rightarrow b$ for some $b \in\mathbb{R}$. 
Set
\begin{equation*}
\tilde{b}=b-a^2.
\end{equation*}

Assume that $a$ is rational and $(m,n_k)$ satisfy
\begin{equation*}
p n_k^2+q|m|+pa=0,
\end{equation*}
then, from \eqref{eq2-e}, we have
\begin{equation*}
\lambda_{mn_k} \rightarrow \tilde{b}, \ \ |m|, k\rightarrow \infty.
\end{equation*}
 Consequently, one can verify that $L^{-1}$ is unbounded if $\tilde{b} = 0$,  $L^{-1}$ is not compact if $\tilde{b} \neq 0$. 
\end{remark}

\begin{proposition}\label{pro2-2}
Let $T$ satisfy \eqref{eq1-4}, and let {\rm (A1)} and {\rm (A2)} hold. If $a$ is an irrational number, then  the range $R(L)$ of $L$ is closed in $L^2(\Omega)$, $N(L)^\perp= R(L)$, $L^2(\Omega)=R(L)\oplus N(L)$ and $L^{-1}\in L(R(L),R(L))$ (the set of bounded linear  operator) is compact, and the following estimations hold:
\begin{eqnarray}
&&\|L^{-1} h\| \leq \frac{1}{\delta} \|h\|, \ \ \forall h\in R(L), \label{eq2-6} \\
&&\langle L^{-1} h, h \rangle \geq -\frac{1}{\gamma} \|h\|^2, \ \ \forall h\in R(L),  \label{eq2-7} \\
&&\|L^{-1} h\|_{L^\infty(\Omega)} \leq C \| h\|,  \ \ \forall h\in R(L)\label{eq2-8},
\end{eqnarray}
where $\delta:= \inf\{|\mu_n-\vartheta^2_m|: \mu_n\neq \vartheta^2_m\}$, $\gamma:=\inf\{\vartheta^2_m -\mu_n: \vartheta^2_m >\mu_n\}$ and $C$ is a constant independent of $h$.
\end{proposition}

\begin{proof}
Let $u_{mn}$ and $h_{mn}$ be respective  the Fourier coefficients of $u$ and $h$, i.e.,
\begin{eqnarray*}
&u=\sum_{m,n}u_{mn}\phi_m(t)\varphi_n(x), \ \ \ &u_{mn}=\int_{\Omega}u\varphi_n\overline{\phi}_m\rho{\rm d}x{\rm d}t,\\
&h=\sum_{m,n}h_{mn}\phi_m(t)\varphi_n(x), \ \ \ &h_{mn}=\int_{\Omega}h\varphi_n\overline{\phi}_m\rho{\rm d}x{\rm d}t.
\end{eqnarray*}

According to the definition of $L$, it follows that the operator equation $Lu=h$ holds if and only if the following equation holds
\begin{equation}\label{eq2-9}
(\mu_n-\vartheta^2_m)u_{mn} = h_{mn}.
\end{equation}

If $u_{mn}$ is a solution of \eqref{eq2-9}, by the definition of $N(L)^\perp$, it follows $h \in N(L)^\perp$. Apparently, we have $N(L)^\perp \subset R(L)$. 
Thus, a necessary condition for the equation \eqref{eq2-9} is  $h \in N(L)^\perp$.

Now, we prove that the condition $h \in N(L)^\perp$ is also sufficient. That is to say, it needs to prove the series
\begin{equation}\label{eq2-10}
\sum_{\mu_n\neq\vartheta^2_m} |u_{mn}|^2=\sum_{\mu_n\neq\vartheta^2_m} \Big{|}\frac{h_{mn}}{\mu_n-\vartheta^2_m}\Big{|}^2
\end{equation}
is convergent, which implies $R(L) \subset N(L)^\perp$.

From Proposition \ref{pro2-1}, $\Lambda(L)$ is an unbounded discrete set, then we have
\begin{equation*}
\delta>0 \ \ {\rm and} \ \ \gamma>0.
\end{equation*}
By \eqref{eq2-10}, we have
\begin{equation*}
\sum_{\mu_n\neq\vartheta^2_m} |u_{mn}|^2=\sum_{\mu_n\neq\vartheta^2_m} \Big{|}\frac{h_{mn}}{\mu_n-\vartheta^2_m}\Big{|}^2 \leq \frac{1}{\delta^2}\sum_{\mu_n\neq\vartheta^2_m} |h_{mn}|^2=\frac{1}{\delta^2}\|h\|^2,
\end{equation*}
which shows the estimation \eqref{eq2-6} holds. Consequently, we get $R(L) = N(L)^\perp$ and $L^2(\Omega)=R(L)\oplus N(L)$.

Noting $\gamma >0$, we have
\begin{equation*}
\langle L^{-1} h, h \rangle = \sum_{\mu_n\neq\vartheta^2_m} \frac{h_{mn}^2}{\mu_n-\vartheta^2_m} \geq -\frac{1}{\gamma} \|h\|^2.
\end{equation*}

According to \eqref{eq2-e}, we have
\begin{equation*}
\lim_{|m|,n\rightarrow \infty}\frac{p^2(\mu_n-\vartheta^2_m)}{p^2n^4-q^2m^2} =1.
\end{equation*}
Since the series
\begin{equation*}
\sum_{m,n} \frac{1}{(p^2n^4-q^2m^2)^2}
\end{equation*}
is convergent, which implies
\begin{equation}\label{eq2-12}
\sum_{m,n} \frac{1}{(\mu_n-\vartheta^2_m)^2}
\end{equation}
is convergent. Therefore, recalling $|\varphi_n(x)| <1$, by Cauchy-Schwarz inequality, it follows
\begin{equation*}
|L^{-1}h(t,x)| \leq \sum_{\mu_n\neq\vartheta^2_m} \Big{|}\frac{h_{mn}}{\mu_n-\vartheta^2_m}\Big{|} \leq \Big{(}\sum_{\mu_n\neq\vartheta^2_m} |h_{mn}|^2\Big{)}^{\frac{1}{2}}\Big{(}\sum_{\mu_n\neq\vartheta^2_m} \frac{1}{(\mu_n-\vartheta^2_m)^2}\Big{)}^{\frac{1}{2}}
\leq C\|h\|,
\end{equation*}
which show \eqref{eq2-8} holding.

At the end of this proposition, we show that the operator $L^{-1}$  is compact. The finite dimensional operator $L^{-1}_N$ is defined by
\begin{equation*}
L^{-1}_N h =\sum_{|m|, n< N} \frac{h_{mn}}{\mu_n-\vartheta^2_m} \ \ {\rm with} \ \ \mu_n\neq\vartheta^2_m, \ \ {\rm for} \ \ N \in \mathbb{N}^+.
\end{equation*}
Therefore, it is sufficient to show
\begin{equation*}
\lim_{N\rightarrow \infty} \|L^{-1}- L^{-1}_N\|_o =0,
\end{equation*}
where $\|\cdot\|_o$ denotes the usual operator norm.
We have
\begin{equation*}
\|L^{-1}- L^{-1}_N\|_o^2 = \sup_{\|h\|=1} \|(L^{-1}- L^{-1}_N) h\| \leq \sum_{|m|, n> N} \frac{1}{(\mu_n-\vartheta^2_m)^2}.
\end{equation*}
Since the series $\sum_{m,n} \frac{1}{(\mu_n-\vartheta^2_m)^2}$ is convergent, it follows
\begin{equation*}
\lim_{N\rightarrow \infty} \sum_{|m|, n> N} \frac{1}{(\mu_n-\vartheta^2_m)^2} =0.
\end{equation*}
Consequently, $L^{-1}$ is compact. We arrive at the result.
\end{proof}

\section{The main result}
\setcounter{equation}{0}
\label{sec:3}

The proof of our main result (Theorem \ref{th3-1}) depends on the following lemma.
\begin{lemma}[Corollary I.2 in \cite{BN78}]\label{lem3-1}
Let $H$ be a real Hilbert space, $L : D(L)\subset H \rightarrow H$ be a closed, dense linear operator with closed range $R(L)$ and $G: H\rightarrow H$ be a monotone demicontinuous (i.e., $G$ is continuous from strong $H$ into weak $H$) operator. Assume

{\rm (i)} $R(L)=N(L)^\perp$;

{\rm (ii)} $L^{-1}: R(L) \rightarrow R(L)$ is  compact;

{\rm (iii)} $\langle Lu, u\rangle\geq -\frac{1}{\gamma} \|Lu\|^2, \ \ \forall u\in D(L)$;

{\rm (iv)} $\|Gu\| \rightarrow \infty$ as $\|u\| \rightarrow \infty$;

{\rm (v)} $ \langle Gu-Gv, u \rangle\geq \frac{1}{\sigma} \|Gu\|^2 -C(v)$ holds for some positive constant $\sigma <\gamma$ and any $u, v\in H$, where $C(v)$ depends only on $v$.

Then, the operator $L+G$ is onto.
\end{lemma}

Obviously,  the function $u\in D(L)$ is a weak solution of the problem \eqref{eq2-1}--\eqref{eq2-3} if and only if
\begin{equation*}
Lu+g(u)=\hat{f},
\end{equation*}
where $\hat{f}= \rho^{-1} f$.

We make the following assumption on $\hat{f}$.

(A3) Let $\hat{f}\in L^\infty(\Omega)$ have the decomposition $\hat{f} = f^* +f^{**}$, where $f^*\in R(L)$ and $\min\{g(-\infty), g(+\infty)\} + \epsilon \leq f^{**} \leq \max\{g(-\infty), g(+\infty)\} - \epsilon$ for a.e. $(t,x)\in\Omega$ and some $\epsilon>0$.

\begin{theorem}\label{th3-1}
Let $T$ satisfy \eqref{eq1-4},  and let {\rm (A1)}, {\rm (A2)} and {\rm (A3)} hold and $a$ be an irrational number. Assume that $g$ is a monotone and continuous function and satisfies $|g(u)|<M$ for some constant $M>0$ and any $u\in R$, then there exists at least one periodic solution to the problem \eqref{eq2-1}--\eqref{eq2-3}.
\end{theorem}

\begin{proof}
Here, we only consider the case where $g$ is non-decreasing. The  case where $g$ is non-increasing is similar to
this case.

Since $g$ does not satisfy  the condition (iv) in Lemma \ref{lem3-1},  we first consider the perturbation equation
\begin{equation}\label{eq3-1}
Lu +G_\varepsilon u= \hat{f}, \ \ \forall \varepsilon>0,
\end{equation}
where $G_\varepsilon u := \varepsilon u+ g(u)$. Clearly, $G_\varepsilon$ is non-decreasing and satisfies
\begin{equation*}
\|G_\varepsilon u\| \rightarrow \infty, \ \ {\rm as} \ \ \|u\| \rightarrow \infty.
\end{equation*}

In what follows, we divide the proof into 4 steps.

\emph{step 1}. Existence of a solution $u_\varepsilon$ of the equation \eqref{eq3-1}.

For any $u,v \in L^2(\Omega)$, a direct calculation yields
\begin{eqnarray*}
\langle G_\varepsilon u- G_\varepsilon v, u-v\rangle &=&\int_\Omega (G_\varepsilon u- G_\varepsilon v)(u-v) \rho{\rm d}t {\rm d}x\\
&=&\int_\Omega (G_\varepsilon u- G_\varepsilon v)\Big{(}\frac{(G_\varepsilon u- G_\varepsilon v) -(g(u)-g(v))}{\varepsilon}\Big{)}\rho{\rm d}t {\rm d}x\\
&=&\frac{1}{\varepsilon}\int_\Omega (G_\varepsilon u- G_\varepsilon v)^2\rho{\rm d}t {\rm d}x -\frac{1}{\varepsilon}\int_\Omega (G_\varepsilon u- G_\varepsilon v)(g(u)-g(v))\rho{\rm d}t {\rm d}x\\
&\geq& \frac{1}{\varepsilon}\Big(\|G_\varepsilon u\|^2 -C(v)\|G_\varepsilon u\|-C(v)\Big),
\end{eqnarray*}
where the last inequality follows from the facts that $|g(u)|<M$ for any $u\in R$ and the continuous embedding $L^2(\Omega) \hookrightarrow L^1(\Omega)$. Therefore, we choose $\varepsilon$ small enough so that $\sigma <\gamma$ and
\begin{equation*}
\langle G_\varepsilon u- G_\varepsilon v, u\rangle \geq \frac{1}{\sigma}\|G_\varepsilon u\|^2 -C(v).
\end{equation*}

Consequently, by Proposition \ref{pro2-2} and Lemma \ref{lem3-1}, there exists a solution $u_\varepsilon \in L^2(\Omega)$ of \eqref{eq3-1} for $\varepsilon$ small enough.

\emph{step 2}. $\|Lu_\varepsilon\|\leq C$ (here and below, $C$ denotes the various constants independent of $\varepsilon$).

Noting $\hat{f} \in L^\infty(\Omega)$ and $f^*\in R(L)$, there exists $w \in L^\infty(\Omega)$ such that $Lw =f^*$. Since $u_\varepsilon$ is a solution of \eqref{eq3-1}, we have
\begin{equation}\label{eq3-a}
\varepsilon u_\varepsilon + L(u_\varepsilon -w)+ g(u_\varepsilon)=f^{**}.
\end{equation}

Moreover, in virtue of $g(-\infty) + \epsilon \leq f^{**} \leq g(+\infty)- \epsilon$, there exists $z\in L^\infty(\Omega)$ such that
$g(z)=f^{**}$. Thus,
\begin{equation}\label{eq3-2}
\varepsilon u_\varepsilon + L(u_\varepsilon -w)+ g(u_\varepsilon)-g(z)=0.
\end{equation}
Taking the inner product of \eqref{eq3-2} with $u_\varepsilon -w$, we have
\begin{equation}
\varepsilon\langle u_\varepsilon, u_\varepsilon \rangle =\varepsilon \langle u_\varepsilon, w\rangle -\langle L(u_\varepsilon -w), u_\varepsilon -w\rangle - \langle g(u_\varepsilon)-g(z), u_\varepsilon -w\rangle.
\end{equation}
Since $g$ is monotone and bounded, by \eqref{eq2-7}, we have
\begin{eqnarray*}
\varepsilon\langle u_\varepsilon, u_\varepsilon \rangle &\leq& \frac{\varepsilon}{2} (\|u_\varepsilon\|^2+\|w\|^2) +\frac{1}{\gamma}\|L(u_\varepsilon-w)\|^2 \\
& \ &  -\langle g(u_\varepsilon)-g(z), u_\varepsilon -z\rangle +\langle g(u_\varepsilon)-g(z), w -z\rangle\\
&\leq& \frac{\varepsilon}{2} (\|u_\varepsilon\|^2+\|w\|^2) +\frac{1}{\gamma}\|L(u_\varepsilon-w)\|^2 +C.
\end{eqnarray*}
By \eqref{eq3-2} and the above estimation, we have
\begin{eqnarray*}
\varepsilon \|u_\varepsilon\|^2 &\leq& \frac{2}{\gamma}\|g(u_\varepsilon)-g(z)+ \varepsilon u_\varepsilon\|^2 +C\\
&\leq& \frac{2\varepsilon^2}{\gamma} \|u_\varepsilon\|^2 +\frac{4\varepsilon C}{\gamma} \|u_\varepsilon\|+C,
\end{eqnarray*}
which implies
\begin{equation}\label{eq3-3}
\varepsilon \|u_\varepsilon\| \leq C.
\end{equation}
Therefore, we have
\begin{equation*}
\|Lu_\varepsilon\| = \|\hat{f} -\varepsilon u_\varepsilon -g(u_\varepsilon)\| \leq \|\hat{f}\| +\varepsilon \|u_\varepsilon\| +\|g(u_\varepsilon)\| \leq C.
\end{equation*}

\emph{step 3}. $\|u_\varepsilon\|_{L^1(\Omega)} \leq C$.

According to $g(-\infty) + \epsilon \leq f^{**} \leq g(+\infty)- \epsilon$, then, for any $\xi \in L^\infty(\Omega)$ with $\|\xi\| \leq \frac{\epsilon}{2}$, there exists $z_\xi \in L^\infty(\Omega)$ with $\|z_\xi\|_{L^\infty(\Omega)} \leq M_\epsilon$
for some $M_\epsilon$ independent of $\xi$ such that
\begin{equation*}
g(z_\xi) =f^{**}+\xi.
\end{equation*}
Substituting $f^{**}= g(z_\xi) -\xi$ into \eqref{eq3-a}, we have
\begin{equation*}
\varepsilon u_\varepsilon + L(u_\varepsilon -w)= - g(u_\varepsilon)+ g(z_\xi) -\xi.
\end{equation*}
Taking the inner product of the above equation with $u_\varepsilon - z_\xi$, we obtain
\begin{eqnarray*}
\langle\varepsilon u_\varepsilon, u_\varepsilon - z_\xi\rangle +\langle L(u_\varepsilon -w), u_\varepsilon - z_\xi\rangle &=& - \langle g(u_\varepsilon)- g(z_\xi), u_\varepsilon - z_\xi\rangle-\langle\xi,u_\varepsilon - z_\xi\rangle\\
&\leq& -\langle\xi,u_\varepsilon - z_\xi\rangle.
\end{eqnarray*}
Thus, by \eqref{eq2-7} and \eqref{eq3-3} we have
\begin{eqnarray*}
\langle\xi, u_\varepsilon\rangle &\leq& -\varepsilon\langle u_\varepsilon,u_\varepsilon \rangle +\varepsilon \langle u_\varepsilon, z_\xi\rangle +\langle \xi, z_\xi\rangle \\
& \ &-\langle L(u_\varepsilon -w), u_\varepsilon - w\rangle-\langle L(u_\varepsilon -w), w - z_\xi\rangle\\
&\leq& \varepsilon\|u_\varepsilon\|\|z_\xi\| +  \frac{\epsilon}{2}\|z_\xi\|_{L^1(\Omega)} + \frac{1}{\gamma}\|L(u_\varepsilon -w)\|^2 + \|L(u_\varepsilon -w)\|\|w-z_\xi\|\\
&\leq& C.
\end{eqnarray*}
Since $\xi \in L^\infty(\Omega)$ with $\|\xi\| \leq \frac{\epsilon}{2}$ is arbitrary,  we get
\begin{equation*}
\|u_\varepsilon\|_{L^1(\Omega)} \leq C.
\end{equation*}

\emph{step 4}. $\|u_\varepsilon\|_{L^\infty(\Omega)} \leq C$.

Decompose $u_\varepsilon = u_{1\varepsilon} + u_{2\varepsilon}$ where $u_{1\varepsilon} \in R(L)$ and $u_{2\varepsilon} \in N(A)$. By Proposition \ref{pro2-2}, it is easy to see  $u_{1\varepsilon} \in L^\infty(\Omega)$.
Noting $\dim N(L) <\infty$, all norms are  equivalent, which shows $u_{2\varepsilon} \in L^\infty(\Omega)$. Thus, $u_{\varepsilon} \in L^\infty(\Omega)$.

\emph{step 5}. Passage to the limit as $\varepsilon \rightarrow \infty$.

Since $L^{-1}$ is compact and  $\dim N(L) <\infty$, noting $\|u_\varepsilon\|_{L^\infty(\Omega)} \leq C$, we can find a sequence $\{\varepsilon_i\}$ satisfying $\varepsilon_i \rightarrow 0$ such that
\begin{eqnarray*}
&u_{1\varepsilon_i} \rightarrow u_1 \  \ {\rm strongly \ in}  \ L^2(\Omega),\\
&u_{2\varepsilon_i} \rightarrow u_2 \ \ {\rm strongly \ in}  \ L^2(\Omega),
\end{eqnarray*}
which implies $u_{\varepsilon_i} \rightarrow u$ strongly in $L^2(\Omega)$ with $u_{\varepsilon_i}:=u_{1\varepsilon_i}+u_{2\varepsilon_i}$ and $u:=u_{1}+u_{2}$.

Since $g$ is monotone, we have
\begin{equation}\label{eq3-5}
\langle g(u_{\varepsilon_i})-g(\psi), u_{\varepsilon_i} -\psi\rangle \geq 0, \ \ \forall \psi\in L^2(\Omega).
\end{equation}
Moreover, since $u_{\varepsilon_i}$ satisfy \eqref{eq3-1}, i.e.,
\begin{equation}\label{eq3-6}
g(u_{\varepsilon_i})=\hat{f} -Lu_{\varepsilon_i} - \varepsilon_i u_{\varepsilon_i}, \ \ \forall \varepsilon>0,
\end{equation}
substituting \eqref{eq3-6} into \eqref{eq3-5}, it follows
\begin{equation*}
\langle \hat{f} -Lu_{\varepsilon_i} - \varepsilon_i u_{\varepsilon_i}-g(\psi), u_{\varepsilon_i} -\psi\rangle \geq 0, \ \ \forall \psi\in L^2(\Omega).
\end{equation*}
Passing to the limit as $i\rightarrow\infty$, we obtain
\begin{equation*}
\langle \hat{f} -Lu - g(\psi), u -\psi\rangle \geq 0, \ \ \forall \psi\in L^2(\Omega).
\end{equation*}
Now, taking $\psi = u-tv$ for any $v \in L^2(\Omega)$, we have $u$ satisfies
\begin{equation*}
Lu + g(u) =\hat{f}.
\end{equation*}
We complete the proof.
\end{proof}

\vskip 5mm
{\bf Acknowledgements}

The author sincerely thanks the anonymous referees for constructive comments
and suggestions.


















\end{document}